\documentclass[reqno,12pt]{amsart}
\usepackage{amsmath,amsthm,amsfonts,amssymb}
\usepackage{enumerate}
\usepackage{amsthm}

\newtheorem{innercustomthm}{Proposition}
\newenvironment{customthm}[1]
  {\renewcommand\theinnercustomthm{#1}\innercustomthm}
  {\endinnercustomthm}

\newtheorem{innercustomle}{Lemma}
\newenvironment{customle}[1]
  {\renewcommand\theinnercustomle{#1}\innercustomle}
  {\endinnercustomle}

\newtheorem{theorem}{Theorem}[section]

\newtheorem{lemma}[theorem]{Lemma}

\begin{document}

\centerline{{\bf Correction to the paper}}

\centerline{{\bf ``Some remarks on Davie's uniqueness theorem''}}

\vskip .2in

\centerline{{\bf
A.V.~Shaposhnikov\footnote{e-mail: shal1t7@mail.ru}}}

\vskip .2in
\centerline{{\bf Abstract}}

\vskip .1in
The property 4 in Proposition 2.3 from the paper ``Some remarks on Davie's uniqueness theorem'' is replaced with a weaker assertion which is sufficient for the proof of the main results. Technical details and improvements are given.



\vskip .2in

\section{Introduction}

We consider the stochastic differential equation
\begin{equation}\label{eq:main}
X_{t} = x + W_{t} + \int_{0}^{t}b(s, X_s)\,ds.
\end{equation}
In the paper \cite{D} the following theorem was proved:
\begin{theorem}\label{th:borel_case}
Let $b: [0, T]\times \mathbb{R}^{d}\mapsto\mathbb{R}^{d}$ be a Borel measurable bounded mapping.
Then for almost all Brownian paths the equation \ref{eq:main} has exactly one solution.
\end{theorem}
In the work \cite{Me} an alternative approach was proposed. 
However as it was pointed out in 
\cite{PR} (see Remark 5.3, p. 24) the uniform H\"older continuity 
(the property 4 from Proposition 2.3 in \cite{Me}) doesn't immediately follow from Kolmogorov continuity theorem and the moments estimates established in
\cite{Me}. Below we present a simple modification of Kolmogorov continuity theorem
and adjust the proofs of the main results from \cite{Me} accordingly. Some other observations regarding the regularity of the flow, in particular, a simple treatment of the case of a bounded drift,
are not included into this short note and will be discussed in a separate paper.

\section{Auxiliary results}
\begin{customthm}{2.3}\label{pr:flow} Let
$$
b \in L^{q}\bigl([0, T], L^{p}(\mathbb{R}^{d})\bigr), \ \frac{d}{p} + \frac{2}{q} < 1.
$$
Then, there exists a H\"older flow of solutions to the equation \ref{eq:main}. More precisely,
for any filtered probability space
$(\Omega, \mathcal{F}, \{\mathcal{F}_{t}\}, P)$ and a Brownian motion $W$, there exists a mapping
$(s, t, x, \omega) \mapsto \varphi_{s, t}(x)(\omega)$ with values in $\mathbb{R}^{d}$, defined for
$0 \leqslant s \leqslant t \leqslant T, \ x \in \mathbb{R}^{d}, \ \omega \in \Omega$,
such that for each $s \in [0, T]$ the following conditions hold:
\begin{enumerate}[{1.}]
\item for any $x \in \mathbb{R}^{d}$ the process $X_{s, t}^{x} = \varphi_{s, t}(x)$ is a continuous $\mathcal{F}_{s,t}$ adapted
solution to the equation \ref{eq:main},
\item $P$-almost surely the mapping $x \mapsto \varphi_{s, t}(x)$ is a homeomorphism,
\item $P$-almost surely for all $x \in \mathbb{R}^{d}$ and $0 \leqslant s \leqslant u \leqslant t \leqslant 1$
$$\varphi_{s, t}(x) = \varphi_{u, t}(\varphi_{s, u}(x)),$$
\item For any 
$\alpha \in (0, 1)$, $\eta >0$, $N > 0$
and a given increasing sequence $S$ of finite sets $\{S_n\}_{n = 0}^{\infty}$ with  
$|S_{n}| \leq 2^{\eta n}$
there exists a set $\Omega'$ of probability $1$
such that for any $s \in S_n$
$x, y \in \mathbb{R}^{d}$ with $|x|, |y| < N, |x -y| \leq 2^{-n}$ and
each $t \in [s, T]$
$$|\varphi_{s, t}(x) - \varphi_{s, t}(y)| \leqslant C(\alpha, T, N, S, \omega)|x - y|^{\alpha}.$$
\end{enumerate}
\end{customthm}
Following the proof given in \cite {Me} we consider the process
$$
Y_t := \psi_{t}(t, X_t) = X_t + U(t, X_t)
$$ 
which is the unique solution of the transformed equation
$$
dY_t = \tilde{b}(t, Y_t)\,dt + \tilde{\sigma}(t, Y_t)\,dW_t,
$$
for details see \cite{Me}.
In the work \cite{Me} the following bound was established:
\begin{equation}
\mathbb{E}\sup\limits_{t \in [0, T]}|Y_{t}^{x} - Y_{t}^{y}|^{a}
\leqslant C(a, T)\bigl(|x - y|^{a} + |x - y|^{a - 1}\bigr),
\end{equation}
It is easy to see that the same arguments provide the estimate
$$
\sup_{s \in [0, T]} \mathbb{E}\sup_{t \in [s, T]} |Y_{s,t}^{x} - Y_{s, t}^{y}|^a \leq 
C(a, T)(|x - y|^a + |x - y|^{a - 1})
$$
Since $\psi_t, \psi_{t}^{-1}$ are Lipschitz continuous uniformly in time an analogous bound holds for $X_{s, t}^x$
$$
\sup_{s \in [0, T]} \mathbb{E}\sup_{t \in [s, T]} |X_{s,t}^{x} - X_{s, t}^{y}|^a \leq 
C(a, T)(|x - y|^a + |x - y|^{a - 1})
$$
We can assume (see  \cite{FF})
that for each $s$ the mapping $X_{s,t}^{x}$ is jointly continuous with respect to $t, x$. To complete the proof we will need the following lemma:
\begin{lemma}\label{le:kolmogorov}
Let $X(s, x)$ be a continuous with respect to $x$ process with values in a complete metric space $(M, \varrho_M)$ on 
$\mathcal{S} \times [0, 1]^d$.
Assume that for some $a, b > 0$
$$ 
\sup_{s \in \mathcal{S}}\mathbb{E} \varrho_{M} (X_{s}(u), X_{s}(v))^{a} \leq |u - v|^{d + b}, \ u, v \in [0, 1]^d
$$
For any $\alpha \in (0, b/a), \eta \in (0, b - \alpha a)$ 
and any increasing sequence $S$
of finite subsets $\{S_n\}_{n = 0}^{\infty}$ with  
$|S_{n}| \leq 2^{\eta  n}$
there exists a set $\Omega'$ of probability $1$ such that
$$ 
\varrho_{M} (X_{s}(u), X_{s}(v)) \leq C(\alpha, \eta, S, \omega)|u - v|^{\alpha} \ s \in S_n,  u, v \in [0, 1]^d, 
|u - v| \leq 2^{-n},
\omega \in \Omega',
$$ 
\end{lemma}
The proof is a minor modification of the standard proof of Kolmogorov continuity theorem, for details see \cite{Ka}.
\begin{proof}
Let $\alpha \in (0, b/a)$. Define $D_n$ as
$$
D_n := \Bigl\{(k_1, \ldots, k_d)2^{-n}; k_1, \ldots, k_d \in \{1, \ldots, 2^n \} \Bigr\}
$$
Let 
$$
Y(s, n) := \max \Bigl\{ \varrho_{M}(X_{s}(u), X_{s}(v)); u, v \in D_n, |u - v| = 2^{-n} \Bigr\}
$$ 
Then
$$
\mathbb{E}(2^{\alpha n} Y(s, n))^a \leq C 2^{\alpha a n} 2^{dn} (2^{-n})^{d + b} \leq
C 2^{(\alpha a - b)n}
$$
Now one readily sees that
$$
\mathbb{E}\sum_{n = 1}^{\infty}\sum_{s \in S_n} (2^{\alpha n} Y(s, n))^a < \infty
$$
Consequently, there exists a set $\Omega'$ of full measure such that 
$$
\sum_{n = 1}^{\infty}\sum_{s \in S_n} (2^{\alpha n} Y(s, n))^a < C(\omega) < \infty, 
\omega \in \Omega'.
$$
in particular
$$
Y(s, n)(\omega) \leq C'(\omega) 2^{-\alpha n}, \ 
s \in S_{n},  \omega \in \Omega'
$$
Using the fact that the sequence $S$ is increasing we obtain the bound
$$
Y(s, m)(\omega) \leq C'(\omega) 2^{-\alpha m}, \ 
s \in S_{n}, m \geq n,  \omega \in \Omega'
$$
Now let $s$ be a fixed point in $S_{n}$.
Applying the standard arguments one can see that
for each $m \geq n$ and any $u, v \in D_m$ such that
$|u - v| \leq 2^{-n}$ the following inequality holds:
$$
\varrho_{M}(X_{s}(u), X_{s}(v)) \leq C''(S, \omega) |u - v|^{\alpha}
$$ 
Now it is easy to complete the proof.
\end{proof}
Now let us come back to the proof of the property~4.
Define a random mapping $J$ from $[0, T] \times [-N, N]^d \time \Omega$
to the Banach space $C([0, T], \mathbb{R}^d)$ equipped with the standard 
sup-norm as follows:
$$
J(\omega, s, x)(t) := X_{s, \min(s + t, T)}^{x}(\omega).
$$
The joint continuity of $X_{s, t}^x$ with respect to $t, x$ immediately implies the mapping $J$
is continuous.
Next, the estimate 
$$
\sup_{s \in [0, T]} \mathbb{E}\sup_{t \in [s, T]} |X_{s,t}^{x} - X_{s, t}^{y}|^a \leq 
C(a, T)(|x - y|^a + |x - y|^{a - 1})
$$
can be written as
$$
\sup_{s \in [0, T]} \mathbb{E} \|J(s, x) - J(s, y)\|^a \leq C(a, T)(|x - y|^a + |x - y|^{a -1})
$$
For any $\alpha \in (0, 1)$ and $\eta > 0$ one can find
sufficiently large $a > 0$ such that
$$
\alpha < \frac{a - 1 - d}{a}, \  \eta < a - 1 - d - \alpha a
$$
so now it is easy to complete the proof applying Lemma \ref{le:kolmogorov}.

\section{\sc Main results}
In this section we adjust the proofs of the main results stated in the paper \cite{Me} using the corrected version of the property 4 from Proposition \ref{pr:flow}.

\begin{theorem}\label{th:holder_case}
Assume that the coefficient $b$ satisfies the following conditions:
\begin{enumerate}[{1.}]
\item there exists $M_{1} \in L^{q_1}\bigl([0, T], \mathbb{R}\bigr)$ such that
$$|b(t, x)| \leqslant M_{1}(t), \ \ t\in[0, T],\ x\in\mathbb{R}^{d}$$
\item there exists $M_{2} \in L^{q_2}\bigl([0, T], \mathbb{R}\bigr)$ and $\beta > 0$ such that
$$
|b(t, x) - b(t, y)| \leqslant M_2(t)|x - y|^{\beta}, \ \ t\in[0, T],\ x,y\in\mathbb{R}^{d}
$$
\item one has
$$
q_1 \geqslant q_2 > 2, \ \beta > 0, \ \frac{\beta}{p_1} + \frac{1}{p_2} > 1, \ \text{where} \
\frac{1}{p_1} + \frac{1}{q_1} = 1, \ \frac{1}{p_2} + \frac{1}{q_2} = 1.
$$
\end{enumerate}
Then there exist a set $\Omega'$ with $P(\Omega') = 1$ such that for each $\omega \in \Omega'$
the equation \ref{eq:main} has exactly one solution.
\end{theorem}
\begin{proof}
Let $Y_t$ be a solution to the equation \ref{eq:main} for a fixed Brownian trajectory $W$.
Then the following estimate holds:
$$
\max\limits_{t \in [0, T]}|Y_t|
\leqslant |x| + \max_{t\in [0, T]}|W_t| + T^{\frac{1}{p_1}}\|M_1\|_{L^{q_1}[0, T]} =: M(x, W),
$$
so without loss of generality we can assume that $b(t, x) = b(t, x)I_{\{|x| < N\}}$ for some $N > 0$.
Then Proposition \ref{pr:flow} (it is clear that one can take $q_1$ for $q$
and any sufficiently large positive number for $p$) yields
that $P$-almost surely the equation \ref{eq:main} has a H\"older-continuous flow
of solutions which will be denoted by
$X(s, t, x, W), \ s \leqslant t, \ x \in \mathbb{R}^{d}.$
$$
1 + \gamma := \frac{\beta}{p_1} + \frac{1}{p_2}, \ \gamma > 0. 
$$
Let us pick $\alpha \in (0, 1)$ such that 
$$
\frac{\alpha\beta}{p_1} + \frac{\alpha}{p_2} = 1 + \delta, \ \delta > 0.
$$

Let us estimate $|Y_r - X(u, r, Y_u, W)|$. It is clear that we have the following trivial bound:
\begin{multline*}
|Y_r -  X(u, r, Y_u, W)| \leqslant \int_{u}^{r}|b(s, Y_s) - b(s, X(u, s, Y_u, W))|\,ds \leqslant\\
\leqslant 2\int_{u}^{r}M_{1}(s)\,ds \leqslant 2\|M_1\|_{L^{q_1}[0, T]}|r - u|^{\frac{1}{p_1}}\\
\end{multline*}
The previous estimate can be improved if we take into account the H\"older-continuity of the coefficient $b$:
\begin{multline*}
|Y_r -  X(u, r, Y_u, W)| \leqslant \int_{u}^{r}|b(s, Y_s) - b(s, X(u, s, Y_u, W))|\,ds \leqslant\\
\leqslant \int_{u}^{r}M_{2}(s)|Y_s - X(u, s, Y_u, W)|^{\beta}\,ds
\leqslant K'\int_{u}^{r}M_{2}(s)|r - u|^{\frac{\beta}{p_1}}\,ds \leqslant\\
\leqslant K'\|M_{2}\|_{L^{q_2}[0, T]}|r - u|^{\frac{\beta}{p_1} + \frac{1}{p_2}}
= K'\|M_{2}\|_{L^{q_2}[0, T]}|r - u|^{1 + \gamma}. \\
\end{multline*}
Define sets $\{S_n\}$ as
$$
S_n := \Bigl\{k/{2^n}; k \in \{0, 1, \ldots, 2^n - 1\}\Bigr\}, \
|S_n| = 2^n
$$  
Using the property $4$ from Proposition \ref{pr:flow}
with $\eta = 1$ and $S = \{S_n\}_{n = 1}^{\infty}$ 
we obtain $\Omega'$ with $P(\Omega') = 1$ 
such that the following estimate holds:
$$
|X(s, t, x, W) - X(s, t, y, W)| \leq
C(\alpha, T, N, \omega)|x - y|^{\alpha}, \
|x - y| \leq \frac{1}{2^n}, s \in S_n
$$

Now let us prove that for each trajectory $W \in \Omega'$ 
the equation
\ref{eq:main} has exactly one solution.
Let us choose a sufficiently large number $K$.
Let $t \in S_{k'}$, where $k' \geq K$. Define an auxiliary function $f$ by the formula
$$
f(s) = X(s, t, Y_s, W) - X(0, t, x, W), \ s \in [0, t].
$$
Let $k \geq k'$ and $u, r$ be of the form
$\frac{i}{2^k}, \frac{i + 1}{2^k}$ respectively,
in particular $u, r \in S_k$.
Recall that
$$
|Y_r -  X(u, r, Y_u, W)| \leq C |r - u|^{1 + \gamma} \leq 
C 2^{-k\gamma}2^{-k}
$$
Since $K$ is supposed to be sufficiently large we may assume that 
$C 2^{-K\gamma} \leq 1$. Consequently,
$$
|Y_r -  X(u, r, Y_u, W)| \leq 2^{-k}
$$
Then
\begin{multline*}
|f(r) - f(u)| = |X(r, t, Y_r, W) - X(u, t, Y_u, W)| = \\
= |X(r, t, Y_r, W) - X(r, t, X(u, r, Y_u, W), W)| 
\leqslant \\
\leqslant 
C(\alpha, S, T, M(x, W), \omega)|Y_r - X(u, r, Y_u, W)|^{\alpha}. \\
\end{multline*}
Finally,
$$
|f(r) - f(u)| \leqslant C(\alpha, S, T, M(x, W), \omega)|r - u|^{1 + \delta}.
$$

Due to the arbitrariness of $k$ we conclude
$$
f(t) = X(x, 0, t, W) - Y_t = 0.
$$
Since $t$ was an arbitrary dyadic number in $[0, 1]$  with  a sufficiently large denominator,
the continuity of $Y_t$ and $X(x, 0, t, W)$ implies the equality $Y_t = X(x, 0, t, W)$ for each $t \in [0, 1]$.
The proof is complete.
\end{proof}

Now we show how to prove uniqueness in the case of a Borel measurable drift 
following \cite{Me}.
Similarly to the proof of Theorem \ref{th:holder_case}, it is readily seen
that without loss of generality we can assume that
$b(t, x) = b(t, x)I_{\{|x| < N\}}$ and $\|b\|_{\infty} \leqslant 1$.

Below we will need the following set of functions:
\begin{multline*}
Lip_{N}\bigl([r, u], \mathbb{R}^{d}\bigr) := \\
:= \bigl\{h \in C\bigl([r, u], \mathbb{R}^{d}\bigr) \ \mid \ |h(t) - h(s)| \leqslant |t - s| \ s,t \in [r, u], \ \max\limits_{s \in [r, u]}|h(s)| \leqslant N \bigr\}
\end{multline*}
with the uniform metric $\varrho(h_1, h_2) = \|h_1 - h_2\|_{\infty}$.

The following result was proved in \cite{Me} and the corresponding arguments remain unchanged.

\begin{customle}{3.6}\label{le:refined_main_borel_estimate}
There exist constants $C, \zeta > 0$, independent of $l = u - r$, and a set $\Omega'$ such that
$$
P(\Omega \setminus \Omega') \leqslant C\exp\Bigl(-l^{-\zeta}\Bigr)
$$
and for any $h_1, h_2 \in Lip_{N}\bigl([r, u], \mathbb{R}^{d}\bigr)$ with 
$\|h_1 - h_2\|_{\infty} \leqslant 4l$, $W \in \Omega'$
the following inequality holds:
$$
|\varphi(h_1, W) - \varphi(h_2, W)| \leqslant Cl^{\frac{4}{3}}.
$$
\end{customle}

We can now proceed to the proof of Theorem \ref{th:borel_case}.

\begin{proof}
Let us fix a positive number $N$.
Let $C, \zeta$ be constants found in Lemma \ref{le:refined_main_borel_estimate}.
For each $k$ we split the interval $[0, 1]$ into $M = 2^k$ closed subintervals
$$
\Bigl[0, \frac{1}{M}\Bigr], \ldots, \Bigl[\frac{M-1}{M}, M\Bigr].
$$
Applying Lemma \ref{le:refined_main_borel_estimate} to each interval
$\bigl[\frac{i}{M}, \frac{i + 1}{M}\bigr]$ we can find the corresponding sets $\Omega_{k, i}$.
Let
$$
\Omega_k := \bigcap_{i = 0}^{M - 1}\Omega_{k, i}.
$$
With the help of the Borel--Cantelli lemma it is easy to show that the set
$$
\Omega' := \liminf_{k \to \infty}\Omega_k = \bigcup_{K = 1}^{\infty}\bigcap_{k = K}^{\infty}\Omega_k
$$
has probability $1$.

Define $S_n$ as
$$
S_n := \Bigl\{k/{2^n}; k \in \{0, 1, \ldots, 2^n - 1\}\Bigr\}, \
|S_n| = 2^n
$$  
Using the property $4$ from Proposition \ref{pr:flow}
with $\eta = 1$ and $S = \{S_n\}_{n = 1}^{\infty}$ 
we may assume (removing, if necessary, a set of zero probability) 
that on the set $\Omega'$  
the following estimate holds:
$$
|X(s, t, x, W) - X(s, t, y, W)| \leq
C(\alpha, T, N, \omega)|x - y|^{\alpha}, \
|x - y| \leq \frac{1}{2^n}, s \in S_n
$$
Let us show that for each $W \in \Omega'$ such that
$$
|x| + \max\limits_{t \in [0, 1]}|W_t| + 1 \leqslant N,
$$
the equation \ref{eq:main} has a unique solution.
Indeed, let $Y_t$ be a solution to the equation \ref{eq:main}.
It is not difficult to see that $|Y_t| \leqslant N$ for each $t \in [0, 1]$.
Due to our choice of $\Omega'$ there exists $K = K(\omega)$ such that for
all $k \geqslant K$ the Brownian trajectory $W$ belongs to $\Omega_k$.
Let
$$
M' = 2^{k'}, \ \ r = \frac{i}{M'}, \ \ \text{where} \ k' \geqslant K.
$$
Let us define an auxiliary function $f$ on the interval $[0, r]$ by the following formula:
$$
f(t) := X(x, 0, r, W) - X(Y_t, t, r, W).
$$

We observe that for any $s \leqslant t$, by the definition of a flow we have
\begin{multline*}
f(t) - f(s) = -X\bigl(Y_t, t, r, W\bigr) + X\bigl(Y_s, s, r, W\bigr) = \\
= -X\bigl(Y_t, t, r, W\bigr) + X\bigl(X(Y_s, s, t, W), r, W\bigr).\\
\end{multline*}
The difference $Y_t - X(Y_s, s, t, W)$ can be represented as follows:
\begin{multline*}
Y_t - X(Y_s, s, t, W) = \\
=\int_{s}^{t}b\Bigl(u, Y_s + W_u - W_s + \int_{s}^{u}b(r, Y_r)\,dr\Bigr)\,du - \\
\int_{s}^{t}b\Bigl(u, Y_s + W_u - W_s + \int_{s}^{u}b(r, X_r)\,dr\Bigr)\,du = \\
= \int_{s}^{t}b\bigl(u, W_u + h_1(u)\bigr)\,du - \int_{s}^{t}b\bigl(u, W_u + h_2(u)\bigr)\,du,
\end{multline*}
where
$$
h_1(u) = Y_s - W_s + \int_{s}^{u}b(r, Y_r)\,dr, \ \ h_2(u) = Y_s - W_s + \int_{s}^{u}b(r, X_r)\,dr.
$$
Let $k \geqslant k'$ и $M = 2^{k}$.
If we take $s, t$ of the form $\frac{i}{M}$ and $\frac{i + 1}{M}$, respectively,
then we obtain the following estimate:
$$
|Y_t - X(Y_s, s, t, W)| \leq \frac{C}{M^{\frac{4}{3}}}
$$
Since we may assume that $M$ is sufficiently large this
inequality implies the bound
$$
|Y_t - X(Y_s, s, t, W)| \leq \frac{1}{M}.
$$
Hence there exists a positive constant $C = C(N, S, W)$ such that
$$
|f(t) - f(s)| \leqslant C|Y_t - X(Y_s, s, t, W)|^{\frac{4}{5}}.
$$
$$
\Bigl|f\Bigl(\frac{i + 1}{M}\Bigr) - f\Bigl(\frac{i}{M}\Bigr)\Bigr| \leqslant \Bigl(\frac{C}{M^{\frac{4}{3}}}\Bigr)^{\frac{4}{5}},
$$
and consequently
$$
|f(r)| \leqslant \frac{C}{M^{\frac{1}{15}}}.
$$
Due to the arbitrariness of $k$ we conclude
$$
f(r) = X(x, 0, r, W) - Y_r = 0.
$$
Since $r$ was an arbitrary dyadic number in $[0, 1]$  with  a sufficiently large denominator,
the continuity of $Y_t$ and $X(x, 0, t, W)$ implies the equality $Y_t = X(x, 0, t, W)$ for each $t \in [0, 1]$.
The proof is complete.
\end{proof}

\centerline{{\bf\it Acknowledgment}}
\vskip .1in
I would like to thank Enrico Priola for pointing out the incompleteness of the proof of the property 4 in Proposition 2.3 in \cite{Me},
fruitful discussions and comments.

\end{document}